\documentclass[12pt]{amsart}
\usepackage{amsmath,amssymb,amsbsy,amsfonts,latexsym,amsopn,amstext,cite,
                                               amsxtra,euscript,amscd,bm,mathabx}
\usepackage{url}
\usepackage[colorlinks,linkcolor=blue,anchorcolor=blue,citecolor=blue,backref=page]{hyperref}
\usepackage{color}
\usepackage{graphics,epsfig}
\usepackage{graphicx}
\usepackage{float} 
\usepackage[english]{babel}
\usepackage{mathtools}
\usepackage{todonotes}
\usepackage{url}
\usepackage[colorlinks,linkcolor=blue,anchorcolor=blue,citecolor=blue,backref=page]{hyperref}
\DeclareMathAlphabet{\mathmybb}{U}{bbold}{m}{n}

\usepackage[norefs,nocites]{refcheck}
\hypersetup{breaklinks=true}

\usepackage[norefs,nocites]{refcheck}
\usepackage[english]{babel}
\begin{document}

\newtheorem{thm}{Theorem}
\newtheorem{lem}[thm]{Lemma}
\newtheorem{claim}[thm]{Claim}
\newtheorem{cor}[thm]{Corollary}
\newtheorem{prop}[thm]{Proposition} 
\newtheorem{definition}[thm]{Definition}
\newtheorem{rem}[thm]{Remark} 
\newtheorem{question}[thm]{Open Question}
\newtheorem{conj}[thm]{Conjecture}
\newtheorem{prob}{Problem}
\newtheorem{Process}[thm]{Process}
\newtheorem{Computation}[thm]{Computation}
\newtheorem{Fact}[thm]{Fact}
\newtheorem{Observation}[thm]{Observation}

\newtheorem{lemma}[thm]{Lemma}

\newcommand{\GL}{\operatorname{GL}}
\newcommand{\SL}{\operatorname{SL}}
\newcommand{\lcm}{\operatorname{lcm}}
\newcommand{\ord}{\operatorname{ord}}
\newcommand{\Op}{\operatorname{Op}}
\newcommand{\Tr}{\operatorname{Tr}}
\newcommand{\Nm}{\operatorname{Nm}}

\numberwithin{equation}{section}
\numberwithin{thm}{section}
\numberwithin{table}{section}

\numberwithin{figure}{section}

\def\sssum{\mathop{\sum\!\sum\!\sum}}
\def\ssum{\mathop{\sum\ldots \sum}}
\def\iint{\mathop{\int\ldots \int}}

\def\wt {\mathrm{wt}}
\def\Tr {\mathrm{Tr}}

\def\SrA{\cS_r\(\cA\)}

\def\vol {{\mathrm{vol\,}}}
\def\squareforqed{\hbox{\rlap{$\sqcap$}$\sqcup$}}
\def\qed{\ifmmode\squareforqed\else{\unskip\nobreak\hfil
\penalty50\hskip1em\null\nobreak\hfil\squareforqed
\parfillskip=0pt\finalhyphendemerits=0\endgraf}\fi}

\def \ss{\mathsf{s}} 

\def \balpha{\bm{\alpha}}
\def \bbeta{\bm{\beta}}
\def \bgamma{\bm{\gamma}}
\def \blambda{\bm{\lambda}}
\def \bchi{\bm{\chi}}
\def \bphi{\bm{\varphi}}
\def \bpsi{\bm{\psi}}
\def \bomega{\bm{\omega}}
\def \btheta{\bm{\vartheta}}

\newcommand{\bfxi}{{\boldsymbol{\xi}}}
\newcommand{\bfrho}{{\boldsymbol{\rho}}}

 \def \xbar{\overline x}
  \def \ybar{\overline y}

\def\cA{{\mathcal A}}
\def\cB{{\mathcal B}}
\def\cC{{\mathcal C}}
\def\cD{{\mathcal D}}
\def\cE{{\mathcal E}}
\def\cF{{\mathcal F}}
\def\cG{{\mathcal G}}
\def\cH{{\mathcal H}}
\def\cI{{\mathcal I}}
\def\cJ{{\mathcal J}}
\def\cK{{\mathcal K}}
\def\cL{{\mathcal L}}
\def\cM{{\mathcal M}}
\def\cN{{\mathcal N}}
\def\cO{{\mathcal O}}
\def\cP{{\mathcal P}}
\def\cQ{{\mathcal Q}}
\def\cR{{\mathcal R}}
\def\cS{{\mathcal S}}
\def\cT{{\mathcal T}}
\def\cU{{\mathcal U}}
\def\cV{{\mathcal V}}
\def\cW{{\mathcal W}}
\def\cX{{\mathcal X}}
\def\cY{{\mathcal Y}}
\def\cZ{{\mathcal Z}}
\def\Ker{{\mathrm{Ker}}}

\def\NmQR{N(m;Q,R)}
\def\VmQR{\cV(m;Q,R)}

\def\Xm{\cX_{p,m}}

\def \A {{\mathbb A}}
\def \B {{\mathbb A}}
\def \C {{\mathbb C}}
\def \F {{\mathbb F}}
\def \G {{\mathbb G}}
\def \L {{\mathbb L}}
\def \K {{\mathbb K}}
\def \N {{\mathbb N}}
\def \PP {{\mathbb P}}
\def \Q {{\mathbb Q}}
\def \R {{\mathbb R}}
\def \Z {{\mathbb Z}}
\def \fS{\mathfrak S}
\def \fB{\mathfrak B}

\def\Fq{\F_q}
\def\Fqr{\F_{q^r}} 
\def\ovFq{\overline{\F_q}}
\def\ovFp{\overline{\F_p}}
\def\GL{\operatorname{GL}}
\def\SL{\operatorname{SL}}
\def\PGL{\operatorname{PGL}}
\def\PSL{\operatorname{PSL}}
\def\li{\operatorname{li}}
\def\sym{\operatorname{sym}}

\def\Mob{M{\"o}bius }

\def\fF{\EuScript{F}}
\def\M{\mathsf {M}}
\def\T{\mathsf {T}}

\def\e{{\mathbf{\,e}}}
\def\ep{{\mathbf{\,e}}_p}
\def\eq{{\mathbf{\,e}}_q}

\def\\{\cr}
\def\({\left(}
\def\){\right)}

\def\<{\left(\!\!\left(}
\def\>{\right)\!\!\right)}
\def\fl#1{\left\lfloor#1\right\rfloor}
\def\rf#1{\left\lceil#1\right\rceil}

\def\Tr{{\mathrm{Tr}}}
\def\Nm{{\mathrm{Nm}}}
\def\Im{{\mathrm{Im}}}

\def \oF {\overline \F}

\newcommand{\pfrac}[2]{{\left(\frac{#1}{#2}\right)}}

\def \Prob{{\mathrm {}}}
\def\e{\mathbf{e}}
\def\ep{{\mathbf{\,e}}_p}
\def\epp{{\mathbf{\,e}}_{p^2}}
\def\em{{\mathbf{\,e}}_m}

\def\Res{\mathrm{Res}}
\def\Orb{\mathrm{Orb}}

\def\vec#1{\mathbf{#1}}
\def \va{\vec{a}}
\def \vb{\vec{b}}
\def \vh{\vec{h}}
\def \vk{\vec{k}}
\def \vs{\vec{s}}
\def \vu{\vec{u}}
\def \vv{\vec{v}}
\def \vz{\vec{z}}
\def\flp#1{{\left\langle#1\right\rangle}_p}
\def\T {\mathsf {T}}

\def\sfG {\mathsf {G}}
\def\sfK {\mathsf {K}}

\def\mand{\qquad\mbox{and}\qquad}

\title[Distribution of sums of square roots]
{Distribution of sums of square roots modulo $1$}

\author[Siddharth Iyer] {Siddharth Iyer}
\address{School of Mathematics and Statistics, University of New South Wales, Sydney, NSW 2052, Australia}
\email{siddharth.iyer@unsw.edu.au}

\begin{abstract}
We improve upon a result of Steinerberger (2024) by demonstrating that for any fixed $k \in \mathbb{N}$ and sufficiently large $n$, there exist integers $1 \leq a_1, \dots, a_k \leq n$ satisfying:
\begin{align*}
0 < \left\| \sum_{j=1}^{k} \sqrt{a_j} \right\| = O(n^{-k/2}).
\end{align*}
The exponent $k/2$ improves upon the previous exponent of $c k^{1/3}$ of Steinerberger (2024), where $c>0$ is an absolute constant. We also show that for $\alpha \in \mathbb{R}$, there exist integers $1 \leq b_1, \dots, b_k \leq n$ such that:
\begin{align*} 
\left\| \sum_{j=1}^k \sqrt{b_j} - \alpha \right\| = O(n^{-\gamma_k}),
\end{align*}
where $\gamma_k \geq \frac{k-1}{4}$ and $\gamma_k = k/2$ when $k=2^m - 1$, $m=1,2,\dots$. Importantly, our approach avoids the use of exponential sums. 
\end{abstract}


\maketitle

\tableofcontents
\section{Introduction}
For a real number $x$ we denote $\|x\|:= \min_{m\in \mathbb{Z}}|x-m|$. Steinerberger \cite{Steinerberger}, through the use of exponential sums, has showed that for a fixed $k \in \mathbb{N}$, there exist integers $1 \leq a_{1},\ldots,a_{k} \leq n$ such that 
\begin{align*}
0<\left\|\sum_{j=1}^{k}\sqrt{a_{j}}\right\| \leq c_{k}\cdot n^{-c k^{1/3}},
\end{align*}
where $c > 0$ is an absolute constant and $c_{k}>0$.
We obtain a stronger and more explicit form of this result without the use of exponential sums.
\begin{thm}
\label{main0}
For every $k \in \mathbb{N}$, there exists an effective constant $C_{k}>0$ so that for every $n\in \mathbb{N}$ there exist integers $1 \leq a_{1},\ldots,a_{k} \leq n$ so that
\begin{align*}
0<\left\|\sum_{j=1}^{k}\sqrt{a_{j}}\right\| \leq C_{k}\cdot n^{-k/2}.
\end{align*}
\end{thm}
We study the distribution of sums of square roots modulo one. Let $\lfloor x\rfloor$ denote the largest integer less than or equal to the real number $x$. The following result is obtained:
\begin{thm}
\label{main}
For every $k \in \mathbb{N}$, there exists an effective constant $D_{k} > 0$ so that for every $\alpha \in \R$ and $n \in \mathbb{N}$ there exist integers \text{ }\text{ }\text{ } $1 \leq b_{1},\ldots,b_{k}\leq n$ with
\begin{align*}
\left\|\sum_{j=1}^{k}\sqrt{b_{j}} - \alpha \right\| \leq D_{k}\cdot n^{-\gamma_{k}},
\end{align*}
where 
\begin{align*}
\gamma_{k} = 2^{\lfloor \log_{2}(k+1) \rfloor - 1} - 1/2 \geq \frac{k-1}{4}.
\end{align*}
\end{thm}
Theorem \ref{main} establishes a bound of $O(n^{-\gamma_{k}})$ on the largest gap in $\sum_{j=1}^{k}\sqrt{b_{j}} \mod 1$, where $b_{j} \in 1,\ldots,n$ and $n \rightarrow \infty$. This result supplements the bound of $O(n^{-2k+3/2})$ for the smallest non-zero gap, established by a direct application of Qian and Wang \cite{Qian}. Notably, their work \cite[Theorem 2]{Qian} implies the existence integers $1\leq u_{j},v_{j} \leq n$ that satisfy
\begin{align*}
0<\left|\sum_{j=1}^{k}\sqrt{u_{j}} -\sqrt{v_{j}}\right| = O(n^{-2k+3/2}).
\end{align*}
We recall for $k = 1$, stronger and more subtle results on the distribution of gaps are known \cite{Gap1, Elkies}. See \cite{Cheng, Steinerberger} for further background in this area of study.

\section{Notation}
\label{Notation}
\begin{itemize}
\item Let $\mathbb{N} = \{1,2,\ldots\}$ be the set of natural numbers.
\item If $x \in \mathbb{R}$ we denote $\lfloor x \rfloor$ , $\lceil x\rceil$ and $\{x\}$ to be the floor, ceiling and integer part of $x$, respectively.
\item Let $U$, $V$, and $W$ be real quantities dependent on parameters $\mathbf{P}$ from some set, along with (possibly) $n \in \mathbb{N}$ and $\alpha \in \mathbb{R}$. We write $U \ll V$, $V \gg U$ or $U = O(V)$ if there exists an effective constant $c > 0$, independent of $n$ and $\alpha$, such that $|U| \leq c\cdot V$ for all choices of $n$ and $\alpha$ (with other variables fixed). Further, we write $U = V+O(W)$ if $U-V \ll W$.
\end{itemize}
\section{Proof of Theorem \ref{main0}}
For a function $f: \mathbb{N}\rightarrow \mathbb{R}$, we let $\Delta_{1}(f;n) = f(n)$ and if $j \in \mathbb{N}$ with $j \geq 2$, we let 
\begin{align*}
\Delta_{j}(f;n) = \Delta_{j-1}(f;n) - \Delta_{j-1}(f;n+1).
\end{align*}
Let $R: \mathbb{N}\rightarrow \mathbb{R}$ satisfy $R(n) = \frac{1}{n}$.

\begin{lem}
\label{RecDiff}
For every $j \in \mathbb{N}$, there exist rational constants $K(j,0),K(j,1),\ldots$ with $K(j,0) >0$ so that if $T \in \mathbb{N}$, we have
\begin{align*}
\Delta_{j}(R;n) = \sum_{t=0}^{T}\frac{K(j,t)}{n^{j+t}}+O(n^{-j-T-1}).
\end{align*}
\end{lem}

\begin{proof}
For $j=1$, this is trivially true. Assume that $j\geq 2$. By induction, we verify that 
\begin{align*}
\Delta_{j}(R;n) = \frac{(j-1)!}{\prod_{m=0}^{j-1}(n+m)}.
\end{align*}
We write
\begin{align*}
\prod_{m=0}^{j-1}(n+m)^{-1} &= n^{-j}\prod_{m=1}^{j-1}\left(1+\frac{m}{n}\right)^{-1}\\
&= n^{-j}\prod_{m=1}^{j-1}\left(\sum_{w=0}^{T}\frac{(-1)^{w}m^{w}}{n^w} + O(n^{-T-1})\right)
\end{align*}
and the result follows.
\end{proof}
For $w \in \mathbb{N}\cup\{0\},$ let $c_{w}$ be the rational number such that 
\begin{align*}
\frac{d^{w}}{dz^w}(z^{1/2}) = c_{w}z^{1/2-w},
\end{align*}
and set
\begin{align*}
C_{w} = \frac{c_{w}}{w!}.
\end{align*}

\begin{lem}
\label{RootSeries}
If $c \in \mathbb{R}$ is a constant, then for every $T \in \mathbb{N}$ and the parameter $n \in \mathbb{N}$, we have
\begin{align*}
\sqrt{n^2+c} = n+\sum_{w=1}^{T}C_{w}c^w n^{-2w+1} + O(n^{-2T-1}).
\end{align*}
\end{lem}
\begin{proof}
The $T$-th order Taylor polynomial of $\sqrt{z}$ at the point $n^2$ is given by 
\begin{align*}
Q_{T,n}(z) = \sum_{w=0}^{T}C_{w}(n^2)^{1/2-w}(z-n^2)^w = n+\sum_{w=1}^{T}C_{w}n^{-2w+1}(z-n^2)^w.
\end{align*}
By the Taylor mean value theorem, we have the estimate
\begin{align*}
\sqrt{z} - Q_{T,n}(z) = \frac{c_{n}z_{n}^{1/2 -T-1}}{(T+1)!}(z-n^2)^{T+1} = C_{T+1}z_{n}^{-\frac{1}{2}-T}(z-n^2)^{T+1},
\end{align*}
where $z_{n}$ lies between $z$ and $n^2$. We now observe that
\begin{align*}
\sqrt{n^2+c} - Q_{y,n}(n^2+c) = O(n^{-2T-1}).
\end{align*}
\par \vspace{-\baselineskip} \qedhere
\end{proof}
\begin{lem}
\label{ReciprocalSeries}
Let $t,d$ and $H$ be integers where $t,H \geq 1$. There exist rational numbers $P(t,d,0),P(t,d,1),\ldots$ with $P(t,d,0)  =1$ and $P(t,d,1) = -dt$ so that
\begin{align*}
\frac{1}{(n+d)^t} = \sum_{q=0}^{H}\frac{P(t,d,q)}{n^{t+q}}+O\left(\frac{1}{n^{t+H+1}}\right)
\end{align*}
as $n \rightarrow \infty$.
\end{lem}
\begin{proof}
Consider the Taylor series of $f(x) = x^{-t}$ at the point $n$ and substitute $n+d$ into the Taylor series. Recall the Taylor mean value theorem.
\end{proof}
If $v \in \mathbb{N}$ and $h \in \mathbb{N}$, we let $\mathcal{S}_{v,h,+}$ to be the set of real-valued functions of the form
\begin{align*}
\sum_{j=1}^{h}a_{j}\sqrt{(n+j-1)^2+\frac{(-1)^{j-1}}{v}},
\end{align*}
where $a_{j} \in \mathbb{Q}$ is some fixed rational number with $a_{j} > 0$ and $n$ is a natural number variable. Similarly, we put $\mathcal{S}_{v,h,-}$ to be the set of real-valued functions of the form
\begin{align*}
\sum_{j=1}^{h}a_{j}\sqrt{(n+j-1)^2+\frac{(-1)^{j}}{v}}.
\end{align*}
We let 
\begin{align*}
\mathcal{S}_{v,h} = \mathcal{S}_{v,h,+}\cup\mathcal{S}_{v,h,-}.
\end{align*}
\begin{lem}
For an element $\omega \in \mathcal{S}_{v,h}$, and an natural number $T \geq 1,$ there exist rational constants $l_{-1}(\omega), l_{0}(\omega), l_{1}(\omega),\ldots$ so that 
\begin{align*}
\omega(n) = l_{-1}(\omega)\cdot n + l_{0}(\omega) + \sum_{t=1}^{T}\frac{l_{t}(\omega)}{n^t} +O(n^{-T-1}).
\end{align*}
\end{lem}
\begin{proof}
By Lemmas \ref{RootSeries} and \ref{ReciprocalSeries}.
\end{proof}
If $h,T \in \mathbb{N}$ and $\varepsilon > 0$, we say that two elements $\omega_{+} \in \mathcal{S}_{v,h,+}$ and $\omega_{-} \in \mathcal{S}_{v,h,-}$ form a $(T,\varepsilon,h)-$pair, if the following hold:
\begin{itemize}
\item $l_{q}(\omega_{\pm}) = 0$ for $q = 1,...,h-1$,
\item $l_{h}(\omega_{+}) = -l_{h}(\omega_{-}) = K(h,0)$,
\item $|l_{h+q}(\omega_{+})-K(h,q)| \leq \varepsilon$ for $q=1,\ldots,T$,
\item $|l_{h+q}(\omega_{-})+K(h,q)| \leq \varepsilon$ for $q =1,\ldots T$.
\end{itemize}
\begin{lem}
\label{Construction}
If $h \in \mathbb{N}$, for every $T \in \mathbb{N}$ and $\varepsilon > 0$, there exists a natural number $b(N,\varepsilon, h)$ so that if $v \geq b(T,\varepsilon, h)$, then there exist $\omega_{+} \in \mathcal{S}_{v,h,+}$ and $\omega_{-} \in \mathcal{S}_{v,h,-}$ so that they are a $(T,\varepsilon,h)-$pair.
\end{lem}
\begin{proof}
We use a proof by induction: By Lemma \ref{RootSeries}, the functions 
\begin{align*}
w_{v,+}(n) = 2v\sqrt{n^2+\frac{1}{v}}
\end{align*}
and
\begin{align*}
w_{v,-}(n)=2v\sqrt{n^2-\frac{1}{v}}
\end{align*}
form a $(T,\varepsilon,1)-$pair, provided $v$ is sufficiently large (as determined by $T$ and $\varepsilon$). Suppose that the proposition has been verified for values $h=1,\ldots,m$.
Given that $r_{+}(n)$ and $r_{-}(n)$ form a $(T+1,\theta,m)-$pair, with the functions belonging to $\mathcal{S}_{v,m,+}$ and $\mathcal{S}_{v,m,-}$ respectively, we see that 
\begin{align*}
A_{+}(n) := r_{+}(n)+r_{-}(n+1) \in \mathcal{S}_{v,m+1,+}
\end{align*}
and
\begin{align*}
A_{-}(n):= r_{+}(n+1)+r_{-}(n) \in \mathcal{S}_{v,m+1,-}.
\end{align*}
We can write
\begin{align*}
A_{+}(n) &= r_{+}(n) + r_{-}(n+1) \\
&= l_{-1}(r_{+}+r_{-})(n)+l_{0}(r_{+}+r_{-})+l_{-1} +\\
&\hspace{1cm}+\sum_{g=0}^{T+1}\frac{K(m,g)}{n^{m+g}}-\sum_{g=0}^{T+1}\frac{K(m,g)}{(n+1)^{m+g}}+\\
&\hspace{2cm}+\sum_{g=1}^{T+1}\frac{\theta_{+,g}}{n^{m+g}}+\sum_{g=1}^{T+1}\frac{\theta_{-,g}}{(n+1)^{m+g}}+O(n^{-m-T-2}),
\end{align*}
where $\theta_{\pm,g}$ are rational constants satisfying $|\theta_{\pm,g}| \leq \theta$. By Lemma \ref{RecDiff}, we continue the above equation as
\begin{align*}
A_{+}(n) &= l_{-1}(r_{+}+r_{-})(n)+l_{0}(r_{+}+r_{-})+l_{-1} +\\
&\hspace{1cm}+\Delta_{m}(R;n)-\Delta_{m}(R;n+1)+\\
&\hspace{2cm}+\sum_{g=1}^{T+1}\frac{\theta_{+,g}}{n^{m+g}}+\sum_{g=1}^{T+1}\frac{\theta_{-,g}}{(n+1)^{m+g}}+O(n^{-m-T-2}).
\end{align*}
Thus by Lemma \ref{RecDiff},
\begin{align*}
A_{+}(n) &= l_{-1}(r_{+}+r_{-})(n)+l_{0}(r_{+}+r_{-})+l_{-1} +\\
&\hspace{1cm}+\sum_{g=0}^{T}\frac{K(m+1,g)}{n^{m+1+g}}+\sum_{g=1}^{T+1}\frac{\theta_{+,g}}{n^{m+g}}+\sum_{g=1}^{T+1}\frac{\theta_{-,g}}{(n+1)^{m+g}}+\\
&\hspace{2.25cm}+O(n^{-x-T-2}).
\end{align*}
By Lemma \ref{ReciprocalSeries}, $l_{q}(A_{+}) = 0$ for $q = 1,\ldots,m$. Furthermore, we record that there exists $E_{m,T} > 0$, independent of $v$, so that 
\begin{align*}
|l_{m+1+g}(A_{+}) - K(m+1,g)| \leq E_{m,T}\cdot\theta \text{ }\text{ }g=0,\ldots,T.
\end{align*}
By a similar analysis we can show $l_{q}(A_{-}) = 0$ for $q = 1,\ldots,m$ and
\begin{align*}
|l_{m+1+g}(A_{-}) + K(m+1,g)| \leq E_{m,T}\cdot\theta \text{ }\text{ }g=0,\ldots,T.
\end{align*}
In particular, if $\theta$ is sufficiently small, $\frac{K(m+1,0)}{l_{m+1}(A_{+})}\cdot A_{+}$ and $-\frac{K(m+1,0)}{l_{m+1}(A_{-})}\cdot A_{-}$ form a $(T,\varepsilon,m+1)-$pair, completing the proof. 
\end{proof}
To finish the argument, by Lemma \ref{Construction} and elementary properties of square roots, there exists $v \in \mathbb{N}$ and positive rational numbers $d_{i}$ so that 
\begin{align*}
\sum_{i=1}^{k}d_{i}\sqrt{(n+i-1)^2+\frac{(-1)^{i}}{v}} = \sum_{i=1}^{k}d_{i}(n+i-1) + \frac{G}{n^{k}} + O(\frac{1}{n^{k+1}}),
\end{align*}
for some non-zero constant $G$. By clearing denominators of rationals, we can show that there exist constants $a_{i} \in \mathbb{N}$ and $b_{i} \in \mathbb{Z}, b_{i} \neq 0$ so that
\begin{align*}
\sum_{i=1}^{k}\sqrt{a_{i}^2(n+i-1)^2+b_{i}} = \sum_{i=1}^{k}a_{i}(n+i-1) + \frac{G_{0}}{n^{k}} + O(\frac{1}{n^{k+1}}),
\end{align*}
for some non-zero constant $G_{0}$.
\section{Proof of Theorem \ref{main}}
Let $\tau \in \mathbb{N}$. We supplement the notation described in Section \ref{Notation} with the following:
\begin{itemize}
\item $p_{i}$ is the $i-$th prime,
\item $P_{\tau} = \{\prod_{j=1}^{\tau}p_{j}^{a_{j}}:a_{j}\in \{0,1\}\},$ so that $P_{\tau}$ is the set of $2^{\tau}$ natural numbers dividing the product $p_{1}\ldots p_{\tau}$,
\item $P_{\tau}^{*} := P_{\tau}\setminus\{1\}$,
\item $S_{\tau} := \{\sum_{f \in P_{\tau}^{*}}c_{f}\sqrt{f}: c_{f}\in \mathbb{Z}\}$,
\item $S_{\tau}^{*} = S_{\tau}\setminus \{0\}$.
\end{itemize}
An important tool in our argument is the following result of Besicovitch \cite[Theorem 2]{Besico} presented in an alternate form.
\begin{lem}
\label{BesicoLem}
Every element $w \in \mathbb{Q}(\sqrt{p_{1}},\ldots,\sqrt{p_{\tau}})$ can be uniquely expressed in the form
\begin{align*}
w = \sum_{f \in P_{\tau}}c_{f}(w)\sqrt{f},
\end{align*}
where the coefficients $c_{f}(w)$ are rational.
\end{lem}

For an element $w \in \mathbb{Z}[\sqrt{p_{1}},\ldots,\sqrt{p_{\tau}}]$ we define $M(w) = \max_{f\in P_{\tau}}|c_{f}(w)|$, which is well defined by Lemma \ref{BesicoLem}.

\begin{lem}
\label{ProductBound}
If $w_{1},w_{2} \in \mathbb{Z}[\sqrt{p_{1}},\ldots,\sqrt{p_{\tau}}]$ then
\begin{align*}
M(w_{1}w_{2})\leq M(w_{1})M(w_{2})\prod_{j=1}^{\tau}(p_{j}+1).
\end{align*}
\end{lem}
\begin{proof}
Let $w \in \mathbb{Z}[\sqrt{p_{1}},\ldots,\sqrt{p_{\tau}}]$, if $j \in \{1,\ldots,\tau\}$ then we compute that $M(\sqrt{p_{j}}w) \leq p_{j}M(w)$. Inductively it follows that if $f \in P_{\tau}$ then $M(w\sqrt{f}) \leq f\cdot M(w)$. Now suppose that $w_{2} = \sum_{f \in P_{\tau}}a_{f}\sqrt{f}$ where $a_{f} \in \mathbb{Z}$, then
\begin{align*}
M(w_{1}w_{2}) &\leq \sum_{f \in P_{\tau}}f|a_{f}|M(w_{1})\\
&\leq \sum_{f \in P_{\tau}}fM(w_{2})M(w_{1}) = M(w_{1})M(w_{2})\prod_{j=1}^{\tau}(p_{j}+1).
\end{align*}
\par \vspace{-\baselineskip} \qedhere
\end{proof}
\begin{lem}
\label{UpperBound}
For every $n \in \mathbb{N}$ there exists $w \in S_{\tau}^{*}$ with $M(w) \leq n$ so that
\begin{align*}
\|w\|\leq \frac{1}{n^{2^{\tau}-1}}.
\end{align*}
\end{lem}
\begin{proof}
Let
\begin{align*}
A_{n} = \left\{\sum_{f \in P_{\tau}^{*}}d_{f}\sqrt{f}: d_{f}\in \{1,\ldots,n\}\right\}.
\end{align*}
By Lemma \ref{BesicoLem} we have $\# A_{n} = n^{2^{\tau}-1}$. After applying the Dirichlet principle to the unit torus $\mathbb{R}/\mathbb{Z}$, the proposition now follows.
\end{proof}
\begin{lem}
\label{LowerBound}
If $n \in \mathbb{N}$ and $w \in \mathbb{Z}[\sqrt{p_{1}},\ldots,\sqrt{p_{\tau}}]$ with $M(w) \leq n$ and $w \neq 0$ then
\begin{align*}
|w| \gg \frac{1}{n^{2^{\tau}-1}}.
\end{align*}
\end{lem}
\begin{proof}
For an element $\mathbf{s}:=(s_{1},\ldots,s_{\tau}) \in \{0,1\}^{\tau}$, the field homomorphism $\sigma_{\mathbf{s}}: \mathbb{Q}(\sqrt{p_{1}},\ldots,\sqrt{p_{\tau}}) \rightarrow \mathbb{Q}(\sqrt{p_{1}},\ldots,\sqrt{p_{\tau}})$ generated by $\sigma_{\mathbf{s}}(\sqrt{p_{j}}) = (-1)^{s_{j}}\sqrt{p_{j}}$ is an automorphism (by Lemma \ref{BesicoLem}). By \cite[Lemma 5.4]{GaloisLem}, we have
\begin{align*}
\text{Gal}(\mathbb{Q}(\sqrt{p_{1}},\ldots,\sqrt{p_{\tau}})/\mathbb{Q}) \simeq (\mathbb{Z}/2\mathbb{Z})^{\tau},
\end{align*}
and thus each element of $\text{Gal}(\mathbb{Q}(\sqrt{p_{1}},\ldots,\sqrt{p_{\tau}})/\mathbb{Q})$ is of the form $\sigma_{\mathbf{s}}$, for some $\mathbf{s} \in \{0,1\}^{\tau}$. Since $w$ is a non-zero algebraic integer, the norm of $w$ in $\mathbb{Q}(\sqrt{p_{1}},\ldots,\sqrt{p_{\tau}})$, denoted by $r$, is a natural number. We have
\begin{align*}
r = \left|\prod_{\sigma \in \text{Gal}(\mathbb{Q}(\sqrt{p_{1}},\ldots,\sqrt{p_{\tau}})/\mathbb{Q})}\sigma(w)\right| = \left|\prod_{\mathbf{s} \in \{0,1\}^{\tau}}\sigma_{\mathbf{s}}(w) \right| \geq 1.
\end{align*}
Since $M(\sigma_{\mathbf{s}}(w)) = M(w) \leq n$, by Lemma \ref{ProductBound} we have
\begin{align*}
|w| \geq \left|\prod_{\mathbf{s} \in \{0,1\}^{\tau}, \ \mathbf{s} \neq \mathbf{0}}\sigma_{\mathbf{s}}(w)\right|^{-1} \gg \frac{1}{n^{2^{\tau}-1}}.
\end{align*}
\par \vspace{-\baselineskip} \qedhere
\end{proof}
\begin{lem}
\label{LowerBound2}
For every $n \in \mathbb{N}$ and $w \in S_{\tau}^{*}$ with $M(w) \leq n$ we have
\begin{align*}
\|w\| \gg \frac{1}{n^{2^{\tau}-1}}.
\end{align*}
\end{lem}
\begin{proof}
With such a choice of $w$, there exists an integer 
\begin{align*}
d \in \left[-n\prod_{j=1}^{\tau}(1+\sqrt{p_{j}}),n\prod_{j=1}^{\tau}(1+\sqrt{p_{j}})\right]
\end{align*}
so that
\begin{align*}
\|w\| = |w - d|.
\end{align*}
It should be noted that $M(w-d) \leq n\left(\prod_{j=1}^{\tau}(1+\sqrt{p_{j}})\right)$, and the result follows from Lemma~\ref{LowerBound}.
\end{proof}
\begin{lem}
\label{AlgorithmicIdea}
 Whenever $\alpha \in [0,1)$ and $n \in \mathbb{N}$, there exists $w_{n} \in S_{\tau}$ satisfying $M(w_{n}) \leq n$ with
\begin{align*}
0 \leq \alpha - \{w_{n}\} \ll \frac{1}{n^{2^{\tau}-1}}.
\end{align*}
\end{lem}
\begin{proof}
Let $j \geq 1$ be an integer, using Lemma \ref{UpperBound} and multiplication by $-1$ if necessary we can find an element $x_{j} \in S_{\tau}^{*}$ which satisfies $M(x_{j}) \leq 2^{j}$ and 
\begin{align*}
0<\{x_{j}\}\leq \frac{1}{(2^{j})^{2^{\tau}-1}}.
\end{align*}
By Lemma \ref{LowerBound2}, we may strengthen the above inequality to
\begin{align*}
\frac{A_{\tau}}{(2^{j})^{2^{\tau}-1}}\leq \{x_{j}\}\leq \frac{1}{(2^{j})^{2^{\tau}-1}}.
\end{align*}
for some $A_{\tau} > 0$.
Form the sequences $(y_{h})_{h=1}^{\infty}$ and $(\alpha_{h})_{h=1}^{\infty}$ inductively as follows; choose $y_{1}$ to be the largest non-negative integer so that 
\begin{align*}
\alpha - y_{1}\{x_{1}\} \geq 0 
\end{align*}
and set $\alpha_{1} = \alpha -y_{1}\{x_{1}\}$. For convenience, we estimate that
\begin{align*}
y_{1} \leq \mu,
\end{align*}
where
\begin{align*}
\mu = \left\lceil \frac{2^{2^{\tau}-1}}{A_{\tau}}\right\rceil.
\end{align*}
Furthermore, we estimate that 
\begin{align*}
0\leq \alpha_{1} \leq \frac{1}{2^{2^{\tau}-1}}.
\end{align*}
Let $y_{h+1}$ be the largest non-negative integer so that 
\begin{align*}
\alpha_{h}-y_{h+1}\{x_{h+1}\} \geq 0.
\end{align*}
Then we set
\begin{align*}
\alpha_{h+1} = \alpha_{h}-y_{h+1}\{x_{h+1}\}.
\end{align*}
In this setting, we observe that
\begin{align*}
0\leq \alpha_{h} \leq \frac{1}{(2^{h})^{2^{\tau}-1}}.
\end{align*}
With $\alpha_{0} = \alpha$ and $h \in 0,1,2,\ldots$ we have
\begin{align*}
0 \leq \alpha_{h+1} = \alpha_{h}-y_{h+1}\{x_{h+1}\} \leq \frac{1}{(2^{h})^{2^{\tau}-1}}-y_{h+1}\frac{A_{\tau}}{(2^{h+1})^{2^{\tau}-1}},
\end{align*}
thus
\begin{align*}
y_{h+1} \leq \mu.
\end{align*}
Now we observe that 
\begin{align}
\label{MajorInequality}
0\leq \alpha_{t} = \alpha - \sum_{h=1}^{t}y_{h}\{x_{h}\} \leq \frac{1}{(2^{t})^{2^{\tau}-1}}.
\end{align}
Put $\omega(t) = \sum_{h=1}^{t}y_{h}x_{h}$. Since $\sum_{h=1}^{t}y_{h}\{x_{h}\} \in [0,1)$ it holds true that
\begin{align*}
\sum_{h=1}^{t}y_{h}\{x_{h}\} = \left\{\sum_{h=1}^{t}y_{h}\{x_{h}\}\right\} = \left\{\sum_{h=1}^{t}y_{h}x_{h} \right\} = \left\{\omega(t)\right\}.
\end{align*}
Thus, by \eqref{MajorInequality} we have
\begin{align}
\label{MajorInequality2}
0 \leq \alpha -\{\omega(t)\} \leq \frac{1}{(2^t)^{2^{\tau}-1}}.
\end{align}
Here we bound 
\begin{align}
\label{MajorInequality3}
M(\omega(t)) = M\left(\sum_{h=1}^{t}y_{h}x_{h}\right) \leq \sum_{h=1}^{t}\mu\cdot 2^{h} \leq \mu \cdot 2^{t+1}.
\end{align}
If $n \geq 4\mu$ we can show that $\omega\left(\left\lfloor\log_{2}(\frac{n}{2\mu})\right\rfloor\right)$ satisfies
\begin{align*}
M\left(\omega\left(\left\lfloor\log_{2}\left(\frac{n}{2\mu}\right)\right\rfloor\right)\right) \leq n
\end{align*}
by \eqref{MajorInequality3}, and
\begin{align*}
0 \leq \alpha - \left\{\omega\left(\left\lfloor\log_{2}(\frac{n}{2\mu})\right\rfloor\right)\right\} &\leq \left(2^{\left\lfloor\log_{2}(\frac{n}{2\mu})\right\rfloor}\right)^{-2^{\tau}+1}\\
&\hspace{1cm}\ll \left(2^{\log_{2}(\frac{n}{2\mu})}\right)^{-2^{\tau}+1} \ll \frac{1}{n^{2^{\tau}-1}}
\end{align*}
by \eqref{MajorInequality2}.
\end{proof}
\begin{lem}
\label{LastOne}
For every $n \in \mathbb{N}$ and $\alpha \in \R$ there exist integers \[0<c_{f}\leq n,\] so that
\begin{align*}
\left\|\sum_{f \in P_{\tau}^{*}}c_{f}\sqrt{f} -\alpha\right\| \ll \frac{1}{n^{2^{\tau}-1}}.
\end{align*}
\end{lem}
\begin{proof}
We can assume that $n \geq 6$. Put 
\begin{align*}
\rho = \sum_{f \in P_{\tau}^{*}}\left\lfloor\frac{n}{2}\right\rfloor\sqrt{f}
\end{align*}
and $\alpha_{0} = \{\alpha -\rho\}$. By Lemma \ref{AlgorithmicIdea} there exists $\kappa \in S_{\tau}$ so that \text{ }\text{ } $M(\kappa) \leq \lfloor n/3\rfloor$ and 
\begin{align*}
0 \leq \alpha_{0} - \{\kappa\} \ll \frac{1}{n^{2^{\tau}-1}}.
\end{align*}
In particular $\kappa+\rho$ is of the form $\sum_{f \in P_{\tau}^{*}}d_{f}\sqrt{f}$ where 
\begin{align*}
\left\lfloor\frac{n}{2}\right\rfloor - \left\lfloor\frac{n}{3}\right\rfloor\leq d_{f} \leq \left\lfloor\frac{n}{2}\right\rfloor+\left\lfloor\frac{n}{3}\right\rfloor,
\end{align*}
and
\begin{align*}
\|\kappa + \rho - \alpha\| \ll \frac{1}{n^{2^{\tau}-1}}.
\end{align*}
\par \vspace{-\baselineskip} \qedhere
\end{proof}
We are finally in a position to prove Theorem \ref{main}. Let
\begin{align*}
\tau = \lfloor \log_{2}(k+1) \rfloor.
\end{align*}
We can assume that
\begin{align*}
n \geq 4 \prod_{j=1}^{\tau}p_{j}.
\end{align*}
Using Lemma \ref{LastOne} we can find integers 
\begin{align*}
0<c_{f} \leq \left\lfloor\frac{\sqrt{n}}{\sqrt{p_{1}\ldots p_{\tau}}}\right\rfloor
\end{align*}
so that 
\begin{align*}
\left\|\sum_{f \in P_{\tau}^{*}}c_{f}\sqrt{f} -\alpha\right\| \ll \left\lfloor\frac{\sqrt{n}}{\sqrt{p_{1}\ldots p_{\tau}}}\right\rfloor^{-2^{\tau}+1} \ll \frac{1}{n^{2^{\tau-1}-1/2}}.
\end{align*}
Let $1=b_{1}=\ldots=b_{k+1-2^{\tau}}$. Then
\begin{align*}
\left\|\sum_{f \in P_{\tau}^{*}}\sqrt{c_{f}^2f} + \sum_{d=1}^{k+1-2^{\tau}}\sqrt{b_{d}} -\alpha\right\|\ll \frac{1}{n^{2^{\tau-1}-1/2}},
\end{align*}
which concludes the proof of Theorem \ref{main}.
\section{Acknowledgements}
\begin{itemize}
\item The author extends heartfelt thanks to Bryce Kerr for his thorough manuscript review. Bryce's insightful suggestion to select $x_j$ with $M(x_j) \leq 2^j$ (rather than $M(x_{j}) \leq j$) in Lemma \ref{AlgorithmicIdea} directly improved the exponent in Theorem \ref{main} from $\frac{-k}{8}+\frac{1}{8}$ to $\frac{-k}{4}+\frac{1}{4}$. Bryce also simplified the proof of Lemma \ref{LowerBound}.

\item The author expresses deep gratitude to Igor Shparlinski for recommending Steinerberger's paper \cite{Steinerberger} and for invaluable insights that simplified the proof of Lemma \ref{RecDiff}.

\item The author thanks Stefan Steinerberger for his continued support and encouragement.

\item Finally, the author acknowledges the support of an Australian Government Research Training Program (RTP) Scholarship, which enabled the completion of this work.
\end{itemize}

\end{document}